\documentclass[12pt]{amsart}
\usepackage{latexsym,fancyhdr,amssymb,color,amsmath,amsthm,graphicx,listings,comment}
\usepackage[section]{placeins}
\newtheorem{thm}{Theorem}

\let\paragraph\subsection

\title{Geodesics for Discrete manifolds}
\author{Oliver Knill}
\date{March 23, 2025}
\address{Department of Mathematics \\ Harvard University \\ Cambridge, MA, 02138 }


\begin{document}
\maketitle

\begin{abstract}
The geodesic flow on a finite discrete q-manifold with or without boundary 
is defined as as a permutation of its ordered q-simplices. This allows to 
define geodesic sheets and a notion of sectional curvature.
\end{abstract}

\section{Geodesic flow}

\paragraph{}
A {\bf finite abstract simplicial complex} is a finite set of non-empty sets $G$ closed under the
operation of taking finite non-empty subsets. 
A {\bf $q$-manifold} is a finite abstract simplicial complex $G$ in which every 
{\bf unit sphere} $S(x)=\delta U(x) = \overline{U(x)} \setminus U(x)$ is a $(q-1)$
sphere, where $U(x)=\{ y \in G, x \subset y \}$ and $\overline{U(x)}$ is the 
closure in the Alexandrov topology in which all stars $\{ U(x), x \in G \}$ form a basis. 
A {\bf $q$-sphere} is a $q$-manifold $G$ such that $G \setminus U(x)$ is contractible for
some $x \in G$ and $G$ is {\bf contractible} if both $S(x)$ and $G \setminus U(x)$ are contractible
for some $x \in G$. As for the foundation of these inductive definitions, the 
empty complex $0=\{\}$ is the $(-1)$-sphere and the 1-point complex $1=\{ \{1\} \}$ is contractible. 

\paragraph{}
If $G$ is a discrete $q$-manifold, the set $\hat{G}$ of $q$-simplices
defines at first only a triangle-free graph with facets as vertices and with
edges $(x,y)$ turned on if $x \cap y$ is a $(q-1)$-simplex in $G$. Let $V(x)$ denote the
set of $0$-dimensional simplices in $x \in G$. The graph $\hat{G}$ upgrades to a 
{\bf discrete cell complex} $\hat{G}$ of dimension $q$ in which the {\bf dual cells}
$\hat{x} = \bigcup_{v \in V(x)} S(v)$ of a $k$-simplex are the
$(q-k-1)$-spheres bounding virtual $(q-k)$-cells in $\hat{G}$. 
The partial order structure of $G$ is dual to the partial order structure of $\hat{G}$ because
$\hat{x} \subset \hat{y}$ if and only if $y \subset x$. The exterior derivative on 
$\hat{G}$ can be defined, so that {\bf Poincar\'e duality}
$b_k=b_{q-k}$ follows from the {\bf Hodge picture for cohomology}, where the {\bf Betti numbers}
$b_k={\rm dim}{\rm ker}(L_k)$, where $L_k$ is the Laplacian on $k$-forms, functions on the set of 
$k$-simplices. 

\paragraph{}
We propose here a definition of geodesic flow and a notion of sectional 
curvature for $q$-manifolds. Both are fundamental notions in differential geometry and pivotal
in general relativity. The naive notion of geodesic as a shortest connection between two points is unsuitable because the 
radius of injectivity is always 2 and because there is no global geodesic flow in general as if
we have an oriented edge, there is no clear way how to continue that edge as in general many continuations
produce locally geodesic arcs of length $2$. Also the naive approach to curvature by taking as sectional 
curvature the Eberhard curvature $1-d(W)/6$ for every embedded wheel graph $W$ with central vertex degree $d(W)$ 
is much too restrictive: positive curvature manifolds are all $q$-spheres or diameter 2 or 3 \cite{SimpleSphereTheorem}. 
We would like to have a chance to prove a {\bf sphere theorem} telling that sufficiently pinched simply connected 
q-manifolds should be spheres. We also want to reproduce the known list of positive curvature manifolds. 

\begin{figure}[!htpb]
\scalebox{0.75}{\includegraphics{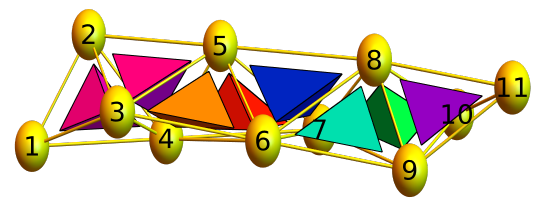}}
\scalebox{0.75}{\includegraphics{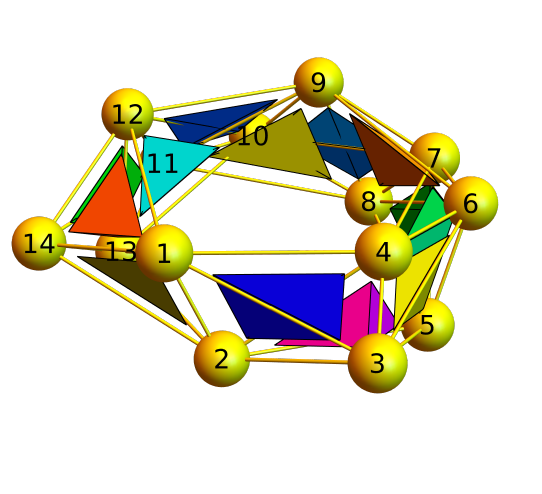}}
\label{definition of the geodesics}
\caption{
The map $T$ maps a totally ordered facet $x$ into its adjacent facet.
To the picture to the left: here $(1,2,3,4)$ is an ordered simplex in a 3-manifold. There exists
exactly one dual simplex $(2,3,4,5)$. The direction of the flow depends
on the orientation. The orientation defines a base point $1$ and an
ordered triple of edges $\{ (12),(13),(14) \}$ that can be thought of as
a basis with base $1$. The geodesic flow transports that basis to
the next simplex $\{ (23),(24),(25) \}$. 
}
\end{figure}

\paragraph{}
If $G$ is a finite abstract simplicial complex, the set $P$ of all {\bf oriented facets} $x$
is a {\bf discrete frame bundle}. For $G=\{ \{1\},\{2\},\{3\},\{1,2\},\{2,3\},\{1,3\} \}$ 
for example, we have $P=\{ \{1,2\},\{2,1\},\{2,3\},\{3,2\},\{1,3\},\{3,1\}\}$. 
Even-so we are in a discrete setting, $P$ is a {\bf principle bundle} with structure group
$S_{q+1}$, the {\bf symmetric group} with $(q+1)!$ elements. With $f_q$ facets, there are $f_q (q+1)!$ points in $P$. 
We think about $x=(x_0,\dots, x_q)$ as a $q$-dimensional {\bf frame attached to $x_0$}
because the total order gives a base point $x_0$ and $q$ basis vectors $(x_0,x_j)$ 
with $1 \leq j \leq q$. The {\bf geodesic flow} is a discrete dynamical systems
$(P,T)$ given by the permutation of $P$ defined by 
$$ T(x_0,x_1,\dots, x_q) = (x_1,x_2, \dots, x_{d-1},x_0') \; , $$
where $x_0'$ is opposite to $x_0$, meaning that 
$\hat{(x_0,x_1, \dots, x_{q-1})}=\{x_0,x_0'\}$ is the boundary 0-sphere of a dual 1-cell.
The permutation $T$ defines a dynamical system on oriented facets in the same way as
the geodesic flow produced a volume preserving reversible flow on the 
unit sphere bundle $SM$ of a manifold $M$. Similarly as the geodesic path in the continuum depends 
also on the initial velocity, the path of the discrete geodesic flow depends on 
the initial permutation. There are $q!$ different directions in which the geodesic can go
as there are $q!$ different frames above every facet $x$. In the above example, the flow is 
$\{1,2\} \to \{2,3\} \to \{3,1\}, \{ 1,2 \}$. 

\begin{figure}[!htpb]
 \scalebox{0.75}{\includegraphics{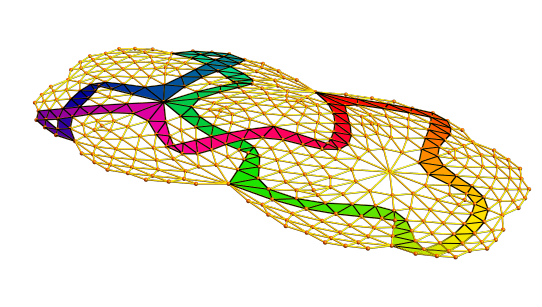}}
 \scalebox{0.75}{\includegraphics{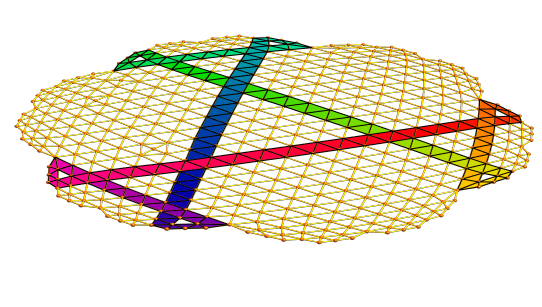}}
 \label{Billiard}
\caption{
When reaching boundary parts, the map does a self rotation on the simplex until a continuation
is possible. It produces a discrete billiard map. The map can be defined for all pure simplicial
complexes for which the stable sphere $S^+(x)=U(x) \setminus \{x\}$ has $1$ or $2$ elements only
for all $(q-1)$-simplices. 
}
 \end{figure}

\paragraph{}
The definition can be extended to {\bf $q$-manifolds with boundary}, in which case the 
geodesic flow defines a {\bf billiard}. If $x$ is at the boundary, meaning that its intersection
is a facet in the boundary, we just let $T$ rotate on itself,
then continue the evolution on $\hat{G}$. So, if the dual $x_0'$ to the wall 
$(x_1,\dots, x_q)$ does not exist, we define 
$$ T(x_0,x_1,\dots, x_q) = (x_1,x_2, \dots, x_{d-1},x_0)  \; . $$
We use this ``self turning" until a continuation is possible. It might not like if a $q$-simplex
is isolated. Extending the definition as such 
allows to define the geodesic flow to a large class {\bf pure simplicial complexes}, complexes of maximal 
dimension $q$ that are generated by the set of its facets = maximal simplices. Note that the 
set $U$ of all facets is an open set in the Alexandrov topology. For a pure simplicial complex,
the closure $\overline{U}=G$. In the language of set theoretical topology on the finite set $G$: 
a pure simplicial complex $G$ has a dense set of facets. 

\paragraph{}
Lets look at the {\bf path complex} $G=\{ \{1,2\},\{2,3\},\{3,4\} \}$ which is a $1$-manifold with boundary. 
The geodesic rule maps $\{1,2\} \to \{2,3\} \to \{3,4\} \to \{4,3\} \to \{ 3,2\} \to \{2,1\} \to \{1,2 \}$. The particle
bounces forth and back on the interval. In the cover $P$, this is a circle. 
An extreme example is if $G$ is generated by a simple $q$-simplex (the Whitney complex of a complete graph $K_{q+1}$),
which can be interpreted $q$-manifold with boundary for which every $(q-1)$ simplex is a boundary face. 
Now $T$ induces a cyclic rotation on $G$. While in dimension $0$, it is a simple fixed point (as $P$ has only one element), 
there are in dimension $q$ already $(q-1)!$ different closed closed geodesics of length $q$ in that single $q$-simplex. 

\paragraph{}
The geodesic flow within a geodesic path $C$ in a $q$-manifold as translation within the path as $C$ can 
be seen as a $q$-manifold with boundary where every $(q-1)$-face is a boundary face. 
The length of a geodesic in a $q$-manifold obtained by 
the number of simplices is up to $q$ times the length of a $1$-dimensional geodesic path in the graph. 
Pushing energy values $\omega(x) = (-1)^{{\rm dim}(x)}$ to the set of simplices of some dimension $k$ 
can define {\bf form curvatures} $K(x)$ on facets \cite{FormCurvatures}. That the form curvature is zero can be 
seen by symmetry. As a closed curve has Euler characteristic $0$ and the curvature is the same on each simplex it must be zero. 
Geodesic arcs can be characterized as curves $C$ on the graph $\hat{G}$ such that 
the induced set of simplices in $G$ has zero curvature except at the end points. 
Closed geodesics have zero curvature everywhere. Note that since $T$ is always a permutation, every geodesic
is closed in the frame bundle $P$. Geodesics "curves" are made up of q-simplices and form tubes. They are q-manifolds 
with boundary which have zero q-form curvature everywhere. This is similar in the continuum: given an arbitrary 
Riemannian 3-manifold $M$ and a geodesic arc $C$, then the boundary of a tube of radius $r$ is a cylinder that has
Gauss curvature going to $0$ for $r \to 0$. 

\begin{thm}
For any q-manifold with or without boundary, there is a globally 
defined geodesic flow given as a permutation on a discrete finite 
frame bundle $P$. The permutation gives a notion of parallel transport. 
Every geodesic path is closed and produces union of facets in $P$ 
which defines a closed simplicial complex with zero curvature. 
\end{thm}

\paragraph{}
The following Mathematica code explains this better than any text. It can be read as pseudo code. 
Each procedure explains what it does. ``Generate" closes a set of sets in the Alexandrov topology, 
Whitney produces a complex from a graph, the walls are the simplices of dimension $q-1$, the facets
re the simplices of dimension $q$. The {\bf open star} is the set $U(x) = \{ y \in G, x \subset y \}$. 
The {\bf star sphere} $S^+(x) = U(x) \setminus \{x\}$ joined with the {\bf core sphere} 
$S^-(x) = \overline{ \{x\}} \setminus \{x\}$ is the {\bf unit sphere} 
$S(x) = \delta(U(x)) = S^+(x) \oplus S^-(x)$, which is a sphere. The map $T$ defines a permutation 
in the frame bundle of $G$. In the example, we take the octahedron complex, which is the Whitney complex
of $K_{2,2,2}$ and platonic solid, then take one of the 4 Catalan solids that are 2-manifolds. 

\begin{tiny}
\lstset{language=Mathematica} \lstset{frameround=fttt}
\begin{lstlisting}[frame=single]
Generate[A_]:=If[A=={},{},Sort[Delete[Union[Sort[Flatten[Map[Subsets,A],1]]],1]]];
Whitney[s_]:=Map[Sort,Generate[FindClique[s,Infinity,All]]];
Walls[G_] :=Select[G,(Length[#]==Max[Map[Length,G]]-1) &];
Facets[G_]:=Select[G,(Length[#]==Max[Map[Length,G]]  ) &];
OpenStar[G_,x_]:=Select[G,SubsetQ[#,x]&];        Stable[G_,x_]:=Complement[OpenStar[G,x],{x}];
Mirror[G_,x_]:=Module[{U=Stable[G,x]},Table[First[Complement[U[[j]],Sort[x]]],{j,Length[U]}]];
T[G_,x_]:=Append[y=Delete[x,1];y,First[Append[Complement[Mirror[G,Sort[y]],{x[[1]]}],x[[1]]]]];
Orbit[G_,x_,n_]:=Module[{S},S[X_]:=T[G,X];NestList[S,x,n]]; 
Orbit[G_,x_]:=Module[{y=x},o={y};While[y=T[G,y];Not[MemberQ[o,y]],o=Append[o,y]];o];
P[G_]:=Module[{F=Facets[G]},Flatten[Table[Permutations[F[[k]]],{k,Length[F]}],1]]; 
T[G_]:=Module[{r=P[G],q={}},While[Length[r]>0,q=Append[q,Orbit[G,r[[1]]]];r=Complement[r,o]];q];
G=Whitney[PolyhedronData["ArchimedeanDual","Skeleton"][[11]]]; Orb=T[G]; Map[Length,Orb]
\end{lstlisting}
\end{tiny}

\paragraph{}
We will next construct geodesic sheets and geodesic curvature for arbitrary q-manifolds. The following 
computes the geodesic curvature for a $(q-2)$-simplex and assumes that the above code has been read. 

\begin{tiny}
\lstset{language=Mathematica} \lstset{frameround=fttt}
\begin{lstlisting}[frame=single]
SecCurvature[G_,x_]:=Module[{Q,r,y,z,l0,l1,l2},
  Q=Select[Stable[G,x],Length[#]==Length[x]+2 &]; l0=Length[Q];
  Sum[y=Q[[m]];r=Complement[y,x];Do[y[[j]]=x[[j]],{j,Length[x]}];
  y[[Length[x]+1]]=r[[1]]; y[[Length[x]+2]]=r[[2]]; z=T[G,y];
  {a,b}=Select[Q,(Length[Intersection[y,#]]==Length[x]+1) &];
  l1=Length[Stable[G,Intersection[y,z,a]]]/2;  
  l2=Length[Stable[G,Intersection[y,z,b]]]/2; 
  1/(3l1)+1/(3l2),{m,l0}]+(2-l0)/6];
Bones[G_]:=Select[G,(Length[#]==Max[Map[Length,G]]-2) &];
SecCurvatures[G_]:=Module[{},B=Bones[G];Table[SecCurvature[G,B[[k]]],{k,Length[B]}]];
Dual[G_]:=Module[{F=Facets[G],e={},n,m},n=Length[F];m=Length[First[F]];
  Do[If[Length[Intersection[F[[k]],F[[l]]]]==m-1,
     e=Append[e,F[[k]]->F[[l]]]],{k,n},{l,k+1,n}]; UndirectedGraph[Graph[e]]];
LocalDisk[G_,x_]:=Module[{Q,Q1,Q2,QQ,y,r,z},
  Q=Select[Stable[G,x],Length[#]==Length[x]+2 &]; QQ=Q;
  Do[y=Q[[m]];r=Complement[y,x];Do[y[[j]]=x[[j]],{j,Length[x]}];
  y[[Length[x]+1]]=r[[1]]; y[[Length[x]+2]]=r[[2]]; z=T[G,y];
  {a,b}=Select[Q,(Length[Intersection[y,#]]==Length[x]+1) &];
  Q1=Stable[G,Intersection[y,z,a]]; QQ=Union[QQ,Q1];
  Q2=Stable[G,Intersection[y,z,b]]; QQ=Union[QQ,Q2],{m,Length[Q]}];
  Subgraph[Dual[G],QQ]];

s=UndirectedGraph[Graph[{1->5,1->7,1->2,1->4,1->8,1->9,2->6,2->9,2->3,2->5,2->10,
 4->3,4->5,4->11,4->7,4->12,8->7,8->14,8->9,8->15,9->10,9->15,3->7,3->10,3->6,
 3->11,5->6,5->12,5->13,10->11,10->15,6->7,6->13,6->14,11->12,11->15,7->14,
 12->13,12->15,13->14, 13->15,14->15}]];      RP2=G=Whitney[s]; 
Orb=T[G] ; Print["Orbit statistics ",Map[Length,Orb]];
Print["Sectional curvatures ",SecCurvatures[G],"Euler ",Total[SecCurvatures[G]]];

G=RP3=Generate[{{1,2,3,4},{1,2,3,5},{1,2,4,6},{1,2,5,6},
{1,3,4,7},{1,3,5,7},{1,4,6,7},{1,5,6,7},{2,3,4,8},
{2,3,5,9},{2,3,8,9},{2,4,6,10},{2,4,8,10},{2,5,6,11},
{2,5,9,11},{2,6,10,11},{2,7,8,9},{2,7,8,10},{2,7,9,11},
{2,7,10,11},{3,4,7,11},{3,4,8,11},{3,5,7,10},{3,5,9,10},
{3,6,8,9},{3,6,8,11},{3,6,9,10},{3,6,10,11},{3,7,10,11},
{4,5,8,10},{4,5,8,11},{4,5,9,10},{4,5,9,11},{4,6,7,9},
{4,6,9,10},{4,7,9,11},{5,6,7,8},{5,6,8,11},{5,7,8,10},{6,7,8,9}}];
Orb=Orbits[G] ; Print["Orbit statistics ",Map[Length,Orb]];
Print["Sectional curvatures ",SecCurvatures[G]];

ToGraph[G_]:=UndirectedGraph[n=Length[G];Graph[Range[n],
  Select[Flatten[Table[k->l,{k,n},{l,k+1,n}],1],(SubsetQ[G[[#[[2]]]],G[[#[[1]]]]])&]]];
Barycentric[s_]:=If[GraphQ[s],ToGraph[Whitney[s]],Whitney[ToGraph[s]]];
Barycentric[s_,n_]:=Last[NestList[Barycentric,s,n]];
G=Barycentric[Whitney[CompleteGraph[{2,2,2,2}]]]; Print["Sec Curvatures ",SecCurvatures[G]];
G=Barycentric[Whitney[CompleteGraph[{2,2,2,2}]],2]; Print["Sec Curvatures ",SecCurvatures[G]];

Cat=PolyhedronData["ArchimedeanDual","Skeleton"];
G=Whitney[Cat[[3]]];   K=SecCurvatures[G];Print[{Union[K],Total[K]}];
G=Whitney[Cat[[4]]];   K=SecCurvatures[G];Print[{Union[K],Total[K]}];
G=Whitney[Cat[[7]]];   K=SecCurvatures[G];Print[{Union[K],Total[K]}];
G=Whitney[Cat[[11]]];  K=SecCurvatures[G];Print[{Union[K],Total[K]}];
T2=Whitney[CirculantGraph[13,Sort[Union[Table[Mod[a^2,13],{a,12}]]]]];
K=SecCurvatures[T2];Print[{Union[K],Total[K]}];
K=SecCurvatures[Barycentric[T2]];Print[{Union[K],Total[K]}];
K=SecCurvatures[Barycentric[T2,2]]; Print[{Union[K],Total[K]}];  (*-1/3,-1/36,0,1/36,1/9*)
\end{lstlisting}
\end{tiny}

\section{Geodesic sheets}

\paragraph{}
Let $G$ be a $q$-manifold, a finite abstract simplicial complex for which every unit sphere is a $(q-1)$-sphere. 
For a $(q-2)$ simplex $x$ in $G$, the dual $\hat{x}$ is the intersection of its
$(q-1)$ vertices and so a $1$-sphere. It defines a cyclic graph with $4$ or more vertices. 
This fact can be seen inductively: for a $0$-simplex $x$, the dual 
$\hat{x}$ is the unit sphere of $x$, a $(q-1)$ sphere. For an edge $\hat{x}$ is a unit sphere in the unit sphere
which is a $(q-1)$ sphere. This continues as such. The dual of a $(q-1)$-simplex $x$ is a $0$-sphere, a 
disjoint union of two points. The dual of a $(q-2)$ simplex $x$ is a $-1$-sphere, 
the empty graph. Back to the dual of a $(q-2)$ simplex, the 
circular graph determines a {\bf virtual 2-cell} in $\hat{G}$ that contributes to the CW
construct. This 2-cell is not physically realized as a sub-complex in $G$, it is 
defined implicitly by $\hat{x}$, similarly as cells in a CW complexes are defined in classical
topology as objects glued to already existing spheres in the build-up. 

\begin{figure}[!htpb]
\scalebox{0.45}{\includegraphics{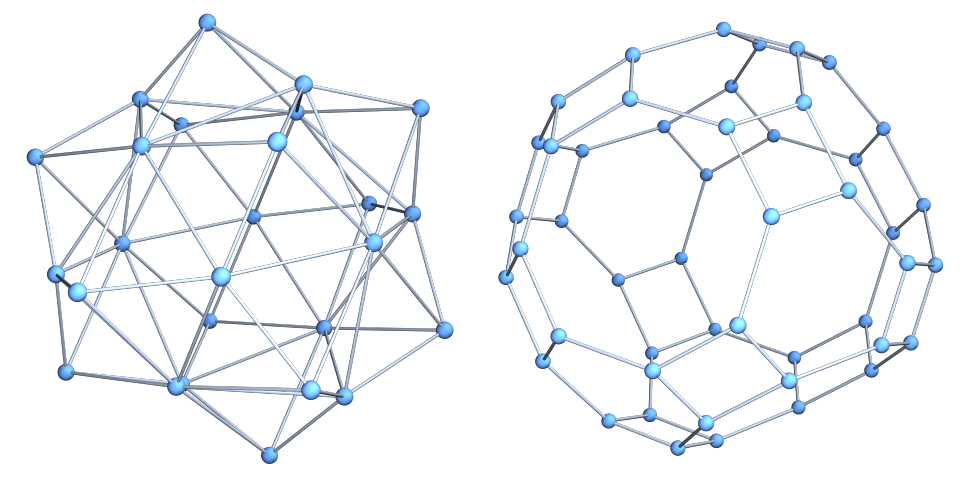}}
\scalebox{0.45}{\includegraphics{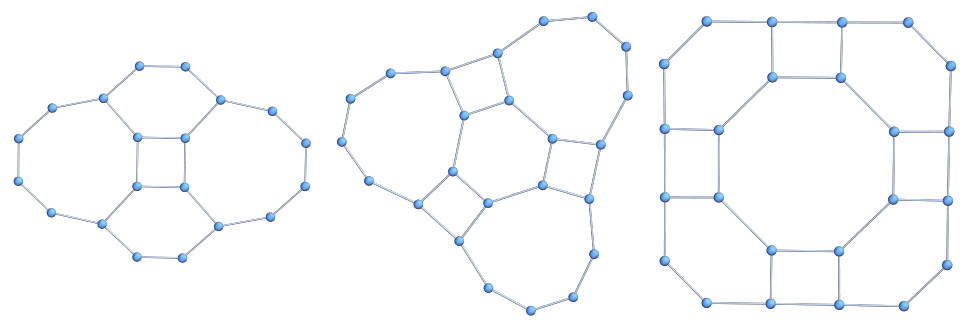}}
\label{Geodesic sheet}
\caption{
We see a 2-manifold $G$, the {\bf greatrhombicuboctahedron}, an example of a 2-sphere. Its
dual $\hat{G}$ is the {\bf disdyakisdodecahedron}. The Eberhard curvatures $1-d(v)/6$ of $G$ are 
$1/3,0,-1/3$. The bones are the q-2=0 simplices in $G$ which are the vertices. There are three
types of {\bf geodesic sheets} defining {\bf sectional curvature}
$(2-4)/6 + (2/3)(2/8+2/6)=1/18$ appearing 12 times,
$(2-6)/6 + (2/3)(3/4+3/8) = 1/12$  appearing 8 times and
$(2-8)/6 + (2/3)(4/4+4/6) = 1/9$ appearing 6 times. Unlike the first order Eberhard curvature, 
these second order curvatures are all positive. 
They are also the {\bf partition curvatures} 
$\frac{(2-m)}{6}+\sum_{i=1}^m \frac{1}{p_i}$ of 
the partitions $p=(8,8,6,6), p=(8,8,8,4,4,4)$ and 
$p=(6,6,6,6,4,4,4,4)$ which encode the possible vertex degrees of 
a vertex and its neighbors. 
The partitions belong to the {\bf unit sphere degree} 28, 36, 40. 
We remark below that if the {\bf unit sphere degree} is less or equal 
than 31, then curvature is automatically positive.  
}
\end{figure}

\paragraph{}
The {\bf geodesic sheet} of a $(q-2)$ simplex  (="bone") 
is globally defined as in the continuum. There is a geodesic 
sheet for every 2 dimensional plane in the tangent space of a point.
Also as in the continuum, where the geodesics can return to the same point
in a different direction, also in the discrete, the sheet becomes multi-valued at 
some time, meaning that if we would start the geodesic flow at a
point on that sheet it would produce an other sheet. Classically, if we involve a 
geodesic beyond the radius of injectivity, the wave front loses the 
injectivity property and so should be considered as evolving in the 
unit tangent bundle.  In the frame bundle $P$, the geodesic sheet 
defines a global $2$ manifold. 

\begin{thm}
For every oriented bone $x$, there is a geodesic sheet. It is a 
$2$-manifold $M_x \subset P$. This manifold allows to define sectional curvature,
a quantity that independent of the orientation choice of $x$.
\end{thm}

\begin{figure}[!htpb]
\scalebox{0.75}{\includegraphics{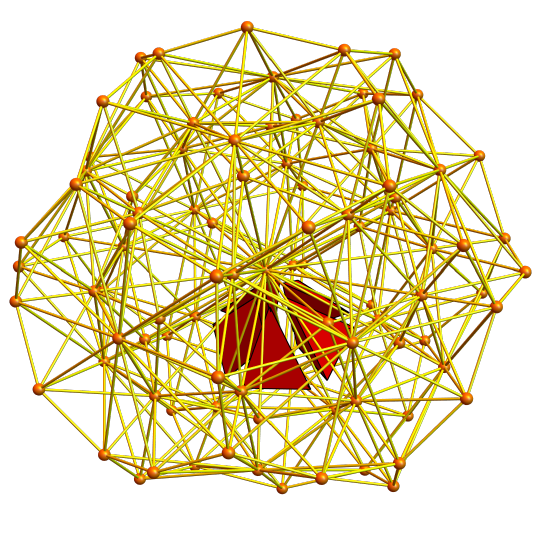}}
\scalebox{0.75}{\includegraphics{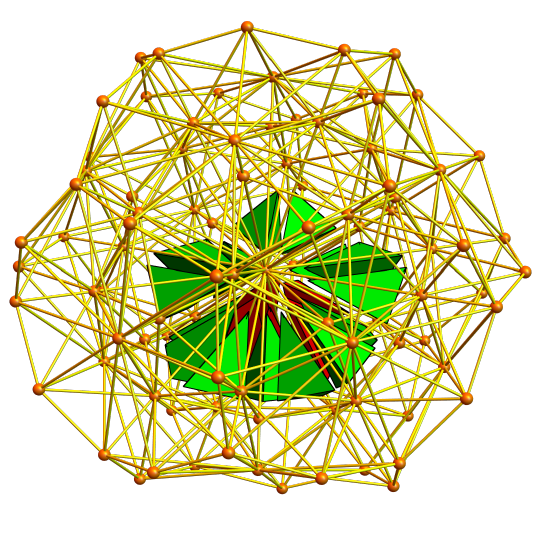}}
\label{Geodesic Sheet}
\caption{
We see a 3-ball with 2-sphere boundary given as the soft Barycentric
refinement of the icosahedron. We chose a random bone $x$ (here an edge
as $q=3$ and $q-2=1$). We first see the dual bone, a set of m=5 
simplices hinging on the bone. We use then the geodesic flow to extend 
this to a discrete two dimensional geodesic sheet which is the dual graph
of a 2-ball. The petal numbers are $p=(6,6,6,6,6)$ in this case. The 
geodesic sheet is made of 20 tetrahedra. 
The sectional curvature is the partition curvature of $p$ which is 
$K(x)=\frac{2-5}{6} + (\frac{3}{2})(\sum_{j=1}^5 \frac{1}{6}) = 1/18$. 
}
\end{figure}

\paragraph{}
The geodesic flow also defines an invertible dynamical system on $(q-1)$-simplices, because every 
oriented $q$-simplex $x$ determines the $(q-1)$-simplex $y=T(x) \cap x$. On the other
hand, an oriented $(q-1)$ simplex $y$ has as a dual $\{ x_0,x_0' \}$ which then 
gives oriented $q$-simplex $x=(x_0,y)$ that satisfies $T(x) =(y,x_0')$. So, the geodesic
flow produces a partition of ordered $(q-1)$-simplices. 
The geodesic sheets on the other hand define a partition of oriented 
$(q-2)$-simplices. 

\paragraph{}
Having been able to define $k=1$ and $k=2$ dimensional geodesic sheets to 
evolve from a $(q-k)$ simplex as starting point, it suggests
to define for any $(q-3)$-simplex a $3$-dimensional 
oriented sheet and in general for any $(q-k)$-simplex a {\bf $k$ dimensional sheet}
in the universal cover $P$ of $G$. For $k=1$, we get the geodesic path 
defined by a 0-sphere (a pair of simplices intersecting in a $(q-1)$-simplex (a wall). 
For $k=3$, where the dual sphere is a $2$-sphere, building evolution is already 
more evolved as we have to build an geodesic at every vertex of the dual bone, build 
all bones attached to it. 

\paragraph{}
A long shot question is then whether in the case of positive curvature and
sufficient pinching of all these sectional curvatures, the complex $G$ must be a 
sphere. In the case of positive curvature all geodesic sheets are spheres. This
does not mean that the manifold is a sphere. It will be interesting to see
what happens in the case of the complex projective plane $G=\mathbb{CP}^2$. 
The geodesic spheres need to be spheres but the three dimensional geodesic spaces
do not foliate this 4 manifold in the same way than in the 4-sphere. There are some
non-contractible spheres as $H^2(G)=1$ (the Betti vector of $G$ is $(1,0,1,0,1)$. 

\paragraph{}
It is interesting in this finite setting that we can continue the geodesic spaces 
for all times and get manifolds in the frame bundle. This is in contrast to the 
continuum, where the geodesic sheets can be very complicated as the frame bundle
has a continuum structure group. 

\section{Sectional curvature}

\paragraph{}
A $(q-2)$-simplex in a simplicial complex is also called {\bf bone}. 
Its {\bf dual bone} $\hat{x}$ is a circular graph in $\hat{G}$, formed by all the facets
hinging on the bone. It can also be seen as the unstable manifold $S^+(x) = U(x) \setminus \{x\}$ of a 
co-dimension 2 simplex whose Barycentric refinement is a circle. 
The join of $S^+(x)$ and simplex in $\hat{x}=S^-(x)$ is the unit sphere 
$S(x)= \delta U(x) = \overline{U(x)} \setminus U(x)$. 

\paragraph{}
We have seen that the form curvature $K(x)$ \cite{FormCurvatures} satisfy 
Gauss-Bonnet $\sum_{x \in B} K(x) = \chi(G)$.
If $\hat{x}$ be the dual of $x$, we can now look at the {\bf $q$-form curvature}
$K(\hat{x}) = \sum_{y \in \hat{x}} K(y)/(q+1)$. 

\paragraph{}
\begin{thm}
If $\hat{x}$ is a bone of length $m$ and $p_k$ are the lengths of the adjacent bones, then 
$K(\hat{x}) = \frac{2-m}{6} + (\frac{2}{3}) \sum_{k=1}^m \frac{1}{p_k}$. 
$\sum_{x \in B} K(\hat{x}) = \chi(G)$. 
\end{thm}
\begin{proof}
Use the Gauss-Bonnet for form curvature on $(q-2)$ simplices.
There are exactly $q+1$ dual cycles that pass through each bone $x \in B$. 
\end{proof} 

\paragraph{}
Example: In the case $q=2$, the bones are the vertices and $\hat{x}$ 
is the unit sphere of $x$ and $K(\hat{x})$ is a third of the
sum of the triangle curvatures of triangles containing $x$. 
For an icosahedron for example, where the triangle curvatures are $1/10$,
we have $K(\hat{x}) = 5*(1/10)/3 = 1/6$. 

\paragraph{}
We can now define {\bf sectional curvature} of $x$ as the curvature 
of $\hat{x}$ within the sheet. Note that curvature depends on the surrounding. 
The geodesic sheet of a bone defined a 2-dimensional complex that is the dual
of a 2-manifold. The sectional curvature was then the 2-form curvature of the 
discrete 2-manifold and this agreed with the Ishida-Higushi curvature for polyhedra. 
The Ishida-Higush curvature itself of a 2-manifold agrees with the Eberhard curvature.
Note that the curvature we introduce here is pushing the curvature back onto the bone.

\paragraph{}
In two dimensions, the bones are the vertices and the curvature is 
$$  K(v) = \frac{2-d(v)}{6} + (\frac{2}{3}) \sum_{k=1}^m \frac{1}{d_k(v)} \; , $$
where  $d(v)$ is the vertex cardinality of $v$ and $d_k(v)$ are the vertex cardinalities of the neighboring vertices. 
The code below shows some examples. The 4 Catalan solids that are 2-manifolds all have positive
curvature, while the naive Eberhard curvature has both positive and negative signs. 

\paragraph{}
The form curvatures satisfy Gauss-Bonnet and agree for 2-dimensional cell complexes with 
Ishida Higuchi curvature \cite{Higuchi} $H(v) = 1-|S(v)|/2 + \sum_{w \in S(v)}  1/|S(w)|$.
As pointed out in \cite{cherngaussbonnet},  had been considered by Gromov \cite{Gromov87}. 
For complexes with triangular faces it agrees with the {\bf Eberhard curvature} $1-|S(v)|/6$.  

\begin{figure}[!htpb]
\scalebox{0.55}{\includegraphics{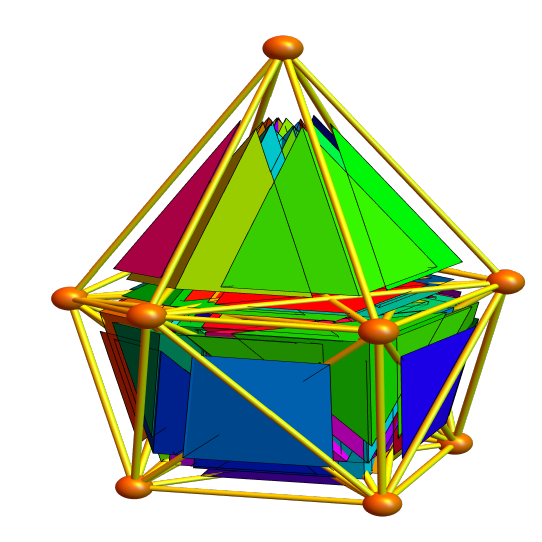}}
\scalebox{0.55}{\includegraphics{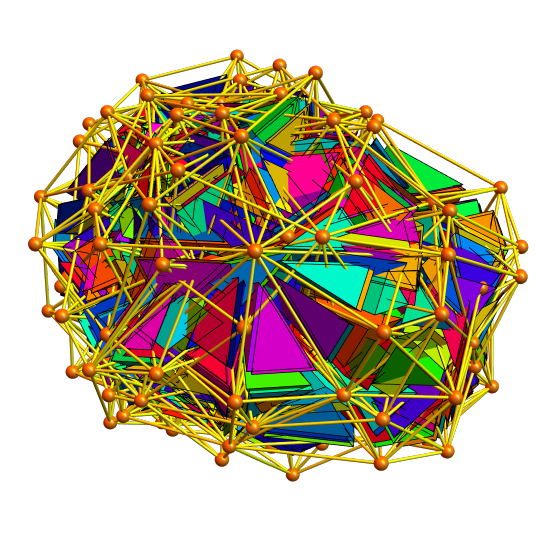}}
\scalebox{0.55}{\includegraphics{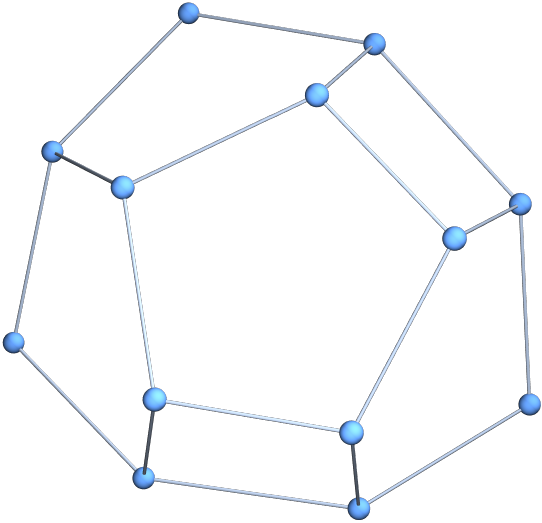}}
\label{RP3 example}
\caption{
A small $3$-manifold $G=\mathbb{R}\mathbb{P}^3$ with $f$-vector 
$f=(11,51,80,40)$ and Euler characteristic $\chi(G)=11-51+80-40=0$. 
Its refinement $G_1$ has $f$-vector $f(G_1)=(182,1142,1920,960)$.
We see one of the possible geodesic sheets of a central bone of order
$5$ surrounded by bones of order $p=(5,5,5,4,4)$, a partition of
$23$. The sectional curvature is 
$(2-5)/6 + (2/3)\sum_{i=1}^5 \frac{1}{p_i} = 7/30$. All sectional 
curvatures of $G$ are either $1/5$ or $7/30$. For the Barycentric
refinement the sectional curvatures are all in $\{1/9, 1/6, 2/9, 1/3\}$. 
Both $G$ and $G_1$ are positive curvature manifolds. 
}
\end{figure}

\section{A curvature for 2-manifolds}

\paragraph{}
As a side product of the sectional curvature investigation we can look independently 
at a curvature for 2-manifolds that is interesting by itself. 
A 2-manifold $G$ is a finite simple graph such that every unit sphere is a 
circular graph with $4$ or more elements. Its Euler characteristic is 
$\chi(G) = |V|-|E|+|F| = \sum_x \omega(x)$, where $\omega(x) = (-1)^{{\rm dim}(x)}$ 
The curvature
$$ K(v) = \frac{2-d(v)}{6} + (\frac{2}{3}) \sum_{w \in S(v)} \frac{1}{d(w)} $$
satisfies Gauss-Bonnet:

\begin{thm}
$\sum_{v \in V} K(v) = \chi(G)$. 
\end{thm} 

\begin{proof} 
First define the form curvature supported on the triangles by moving all the 
energies $\omega(x)$ to the triangles $t=(a,b,c)$. This was the Ishida-Higushi
curvature.  This gives $K(t) = 1-3/2 + (\frac{1}{a} + \frac{1}{b} + \frac{1}{c})$
for every triangles. Then distribute the energy from the triangles to the vertices
by moving $K(t)/3$ from each triangle to its vertices. 
\end{proof} 

\paragraph{}
Here is some code which explains this. It is self-contained and does not need the above part. 

\begin{tiny}
\lstset{language=Mathematica} \lstset{frameround=fttt}
\begin{lstlisting}[frame=single]
Generate[A_]:=If[A=={},{},Sort[Delete[Union[Sort[Flatten[Map[Subsets,A],1]]],1]]];
Whitney[s_]:=Map[Sort,Generate[FindClique[s,Infinity,All]]]; 
Curvature[s_,v_]:=Module[{G=Whitney[s]},
  U=Select[G,(MemberQ[#,v]&&Length[#]==3)&];Sum[{a,b,c}=U[[j]];
  k=VertexDegree[s,a]; l=VertexDegree[s,b];m=VertexDegree[s,c];
  1/k+1/l+1/m -1/2,{j,Length[U]}]/3];
Curvatures[s_]:=Module[{V=VertexList[s]},Table[Curvature[s,V[[k]]],{k,Length[V]}]];
Cat=PolyhedronData["ArchimedeanDual","Skeleton"];
c3=Cat[[3]];  K=Curvatures[c3];Print[{Union[K],Total[K]}];
c4=Cat[[4]];  K=Curvatures[c4];Print[{Union[K],Total[K]}];
c7=Cat[[7]];  K=Curvatures[c7];Print[{Union[K],Total[K]}];
c0=Cat[[11]]; K=Curvatures[c0];Print[{Union[K],Total[K]}];
M=13;  torus=CirculantGraph[M,Sort[Union[Table[Mod[a^2,M],{a,M-1}]]]];
K=Curvatures[torus];Print[{Union[K],Total[K]}]
\end{lstlisting}
\end{tiny}

\section{Partition curvature} 

\paragraph{}
This motives to look at something independent to manifolds and to 
number theory. We can define the {\bf partition curvature} 
$K(p)$ of a {\bf partition} $p=(p_1, \dots, p_m)$ as
$$ K(p) = \frac{2-m}{6} + (\frac{2}{3}) \sum_{k=1}^m \frac{1}{p_k} \; .  $$
We only look at cases with $m>3,p_k>3$. 

\paragraph{}
This also puts the finger on the {\bf unit sphere vertex degree}
$$   \sum_{w \in S(v)} d(w) \;  $$  
of a graph.  The following is a handy test to establish positive curvature.

\begin{thm}[31 theorem]
If $\sum_{w \in S(v)} d(w) \leq 31$, then $K(v)>0$ independent of the 
degree of $v$. If $\sum_{w \in S(v)} d(w)=32$, then $K(v) \geq 0$ with
equality for $S(v)=4,S(w)=8$ for all $(v,w) \in E$. 
\end{thm}

\begin{proof}
This is just checking all cases. This can be done with a few lines.
\end{proof}

\begin{tiny}
\lstset{language=Mathematica} \lstset{frameround=fttt}
\begin{lstlisting}[frame=single]
PartitionCurvatures[n_]:=Module[{p=Select[IntegerPartitions[n],(Min[#]>3&&Length[#]>3)&]},
 Table[{p[[j]],(2/3) Sum[1/p[[j,k]],{k,Length[p[[j]]]}]+(2-Length[p[[j]]])/6},{j,Length[p]}]];
Min[Sort[PartitionCurvatures[31]]]
Select[PartitionCurvatures[32], #[[2]] == 0 &]
Select[PartitionCurvatures[33], #[[2]] < 0 &]
\end{lstlisting}
\end{tiny}

\paragraph{}
For $\sum_{w \in S(v)} d(w) \leq 33$ there are 4 cases
$K(10, 9, 7, 7)=-(2/945), K(10, 8, 8, 7)=-(1/210)$, 
$K(9,9,8,7)=-(5/756)$ and $K(9, 8, 8, 8)= -(1/108)$. 
This suggests to look at the {\bf total sphere vertex degree} $d(S(v))= \sum_{w \in S(v)} d(w)$. 
If the total sphere vertex degree is everywhere smaller than 32, we have positive curvature. 

\paragraph{}
We can use this for sectional curvature. If the sum of the neighboring bone degrees
of a bone in a d-manifold is smaller than 32 for all bones, then we have a positive
curvature manifold is smaller than 32 for all bones, then we have a positive
curvature manifold is smaller than 32 for all bones, then we have a positive curvature manifold.

\paragraph{}
We can ask many questions like: 

\begin{itemize}
\item Which partition spectra are possible for q-manifolds?  We do not know yet how to 
characterize the rational numbers or whether all rational numbers can occur. 
\item Is it true that we can realize any 2-manifold as a geodesic sheet of a higher dimensional 
manifold? Which 3 manifolds can be realized as three dimensional sheets in a 4 manifold? 
\item Under which conditions is positive curvature invariant under Barycentric refinement? 
Here is an example showing that the answer is "no". If $G$ is the octahedron complex
then its curvature is constant $1/3$ and agrees with the Eberhard curvature. 
After a Barycentric refinement, we get the spectrum $\{1/18,1/12,1/9 \}$ which 
is still a positive curvature 2-sphere. 
But the next refinement produces the spectrum $\{-1/9, -1/72, 0, 1/72, 1/36, 1/9 \}$. 
The positive curvature is gone. 
\item Is there a bound for $n=\sum_{i=1}^m p_i$ of a partition
after which we negative curvature? The example $p=(n-12,4,4,4)$ shows that the answer is no. 
Even if ask all vertex degrees to be larger than 4, then there always only 
4 cases  $(n,x,5,5,5)$, with $x=5,6,7$  or $(n-20,5,5,5,5)$. 
It is also not true that for large enough $m$, we have always negative curvature:
$(2-m)/6 + (2/3) \sum_{i=1}^m 1/p_i \leq 1/3-m/6+ (2/3) m/4 = 1/3$ shows that if 
all neighbors of a degree $m$ vertex have degree $4$ then the curvature is $1/3$. 
However, if no vertex degrees 4 are allowed, then for $m>10$, we always have negative
curvature. The reason is $1/3-m/6+ (2/3) m/5 = 1/3-m/30$. 
\end{itemize}

\paragraph{}
If all vertex degree have to be larger than 5, there are no cases.
If only the central vertex degree can be 5, there are a few cases $x6666$ with $x=6,7,8,9,10,11$.
                                                                  Then there is $x7666$ with $x=7,8,9$
If we have an isolated 5 and the unit sphere vertex degree is larger than 35, we have no positive curvature. 

\section{Poincare Hopf}

\paragraph{}
Curvature is what happens if all energies $\omega(x) = (-1)^{{\rm dim}(x)}$ contributing to 
{\bf Euler characteristic} $\chi(G) = \sum_{x \in G} \omega(x)$ are pushed to using a specific
rule to the vertices $V=V(G) = \{ x \in G, {\rm dim}(x)=0 \}$ \cite{poincarehopf} or 
to some higher dimensional part of space \cite{FormCurvatures}. If the energies $\omega(x)$ are
pushed in a deterministic way to the vertex $v \in x$, on which a function $g: V \to \mathbb{R}$ is
minimal, we get Poincar\'e-Hopf. Since the curvature is not split, the result is a divisor, an 
integer valued function on the vertices. This rule does not necessarily have to come from a 
gradient field, it can just be a rule $x \in G \to v=F(x) \in V$ with $v \in x$ which we can consider
to be a vector field. If we take a probability space of such rules and average the index 
we get index expectation curvatures $K(v)={\rm E}[i_g(v)]$ or $K(v)={\rm E}[i_F(v)]$ \cite{indexexpectation}. 
The Levitt curvature $\sum_{k={-1}} (-1)^k f_k(S(x))/(k+2)$ is such an average. For 2-dimensional 
manifolds, this curvature is the Eberhard curvature $K(v) = 1-d(v)/6$. 

\paragraph{}
Can one see the second order curvature discussed here in the context of sectional curvature
as an index expectation. Yes, this can be done general, whenever we 
can look at facet curvatures (which works if the set of facets cover the simplicial complex). 
Start with a map $F_1$ that maps every simplex to a triangle containing it, then 
chose a map $F_2$ from triangles to vertices. This produces a map
$T:V \to V$ on vertices $V$ which can again be seen as a vector field. (Now this is more natural
as this notion of vector field also defines a dynamical system). 
We can push the energies of simplices to the vertices. This can be
done in any dimension provided that the maximal simplices cover the graph. 
It now defines Poincar\'e-Hopf indices on the vertices which add up to 
Euler characteristic. 

\begin{tiny}
\lstset{language=Mathematica} \lstset{frameround=fttt}
\begin{lstlisting}[frame=single]
Generate[A_]:=If[A=={},{},Sort[Delete[Union[Sort[Flatten[Map[Subsets,A],1]]],1]]];
Whitney[s_]:=Generate[FindClique[s,Infinity,All]];
ToGraph[G_]:=UndirectedGraph[n=Length[G];Graph[Range[n],
Select[Flatten[Table[k->l,{k,n},{l,k+1,n}],1],(SubsetQ[G[[#[[2]]]],G[[#[[1]]]]])&]]];

s=CompleteGraph[{2,2,2}]; s= ToGraph[Whitney[s]]; s=ToGraph[Whitney[s]]; 
G=Whitney[s]; G0=Select[G,Length[#]<=3 &]; G1=Select[G,Length[#]==1 &];
G2=Select[G,Length[#]==2 &]; G3 = Select[G,Length[#]==3 &];
F1=Table[G0[[k]]->RandomChoice[Select[G3,SubsetQ[#,G0[[k]]] &]],{k,Length[G0]}];
F2=Table[G3[[k]]->RandomChoice[Select[G1,SubsetQ[G0[[k]],#] &]],{k,Length[G3]}];
W0=Table[-(-1)^Length[G0[[k]]],{k,Length[G0]}]; W1=W2=Table[0,{Length[G0]}];
Do[x=G0[[k]];W1[[Position[G0,x /.F1][[1,1]]]]+=W0[[k]],{k,Length[G0]}];
Do[x=G0[[k]];W2[[Position[G0,x /.F2][[1,1]]]]+=W1[[k]],{k,Length[G0]}];
{Total[W0], Total[W1], Total[W2]}

(* The two random maps define a self map on vertices. Lets look at the graph *)
T=Table[x=G1[[k]]; x-> (x /. F1 /. F2),{k,Length[G1]}];
Graph[T]
\end{lstlisting}
\end{tiny}

\paragraph{}
Here implement this by choosing random functions on vertices and facets:

\begin{tiny}
\lstset{language=Mathematica} \lstset{frameround=fttt}
\begin{lstlisting}[frame=single]
Generate[A_]:=If[A=={},{},Sort[Delete[Union[Sort[Flatten[Map[Subsets,A],1]]],1]]];
Whitney[s_]:=Generate[FindClique[s,Infinity,All]]; w[x_]:=-(-1)^Length[x];
ToGraph[G_]:=UndirectedGraph[n=Length[G];Graph[Range[n],
  Select[Flatten[Table[k->l,{k,n},{l,k+1,n}],1],(SubsetQ[G[[#[[2]]]],G[[#[[1]]]]])&]]];
Barycentric[s_]:=ToGraph[Whitney[s]]; Euler[s_]:=Total[Map[w,Whitney[s]]];

Index[s_]:=Module[{G,V,F,(* f,g,*)W,J,q}, (* index with functions on vertices/facets*)
G=Whitney[s]; q=Max[Map[Length,G]];       (* clique number = dim(G)+1 *)
V=Flatten[Select[G,Length[#]==1&]]; (*vertices*) F=Select[G,Length[#]==q&];  (* facets *)
f=Table[V[[k]]->Random[],{k,Length[V]}]; g=Table[F[[k]]->Random[],{k,Length[F]}];
minf[x_]:=x[[First[Flatten[Position[x /. f,Min[x /. f]]]]]]; (* minimum position of f on x *)
ming[x_]:=Module[{S=Select[F,SubsetQ[#,x] &]},S[[First[Flatten[Position[S/.g,Min[S/.g]]]]]]];
W=Table[G[[k]] -> -(-1)^Length[G[[k]]],{k,Length[G]}];   (* primordial energy *)
J=Table[0,{Length[V]}]; Do[x=G[[k]];J[[minf[ming[x]]]] += (x /. W) ,{k,Length[G]}]; J]  

s=Barycentric[PolyhedronData["Icosahedron","Skeleton"]];i=Index[s];
Print[{i,Total[i],Euler[s]}];
s=Barycentric[CompleteGraph[{2,2,2}]]; i=Index[s]; 
Print[{i,Total[i],Euler[s]}];
\end{lstlisting}
\end{tiny}

\section{More remarks and questions}

\paragraph{}
To summarize, we have defined a geodesic flow on discrete generalized q-manifolds $G$ as a finite permutation
on a principle bundle $E$ of the dual manifold $\hat{G}$. The reason were obvious shortcomings
from the usual geodesic notion. There is no way we can on a discrete manifold give a natural 
dynamics on edges for example. In rare cases this is possible. 
For Eulerian 2-manifolds one can define a geodesic flow as well as 
a billiard \cite{knillgraphcoloring3}. 

\paragraph{}
The space on which the flow runs is a frame bundle, the set of oriented facets $x$. The order of a simplex 
plays the role of the frame. The geodesic flow uses a discrete "moving frame" terminology. In 
some sense we are close to Cartan's 1937 approach to geometry. 
Through every $x$, a geodesics can start into $(q+1)!$ directions
but as in the continuum there can be self-intersections as in the continuum.
Every $(q-2)$-simplex $x$ defines a $2$-dimensional manifold $M(x)$ that is embedded in $G$ and
immersed in the bundle $E$. The sectional curvature $K(x)$ is defined is then a second 
order curvature.

\paragraph{}
Since the setup can be extended to manifolds with boundary we can study discrete
billiards. We can search for example for {\bf ergodic billiards}, billiards where the geodesic flow
has a single cycle in the frame bundle $P$. 
We have also seen that the geodesic flow can be defined for a large class of pure simplicial 
complexes. 

\paragraph{} 
We could write the dynamics $T$ as a product of involutions by
looking at open and closed points in $\hat{G}$. This is a larger theme 
\cite{GraphsGroupsGeometry}. The geodesic dynamics
is an action of the cyclic group $\mathbb{Z}_t$ where $t$
is the GCD of all periods of the geodesics. It can also be seen as an action of 
the dihedral group $\mathbb{D}_t$. To do so, double $\hat{G}$, where the first 
copy consists of $x$ as an open set and the second $B(x)$ as the closure. 
Now define on the open part
$$ A(x) = (x_1, \dots, x_{d-1},x_0') $$
and declare it closed on the closed part
$$ A(x) = (x_0',x_2, \dots, x_d) $$
and declare it closed. Now $A^2=Id$. With the additional involution $B$ switching open 
and closed sets, we have then $T=B \circ A$ and $T^{-1} = A \circ B$. 

\bibliographystyle{plain}

\end{document}